\documentclass[10pt]{article}

\usepackage[a4paper,margin=24mm]{geometry}
\usepackage[T1]{fontenc}
\usepackage[utf8]{inputenc}
\usepackage{lmodern}
\usepackage{microtype}
\usepackage{amsmath,amssymb,amsthm,mathtools}
\usepackage{booktabs,array}
\usepackage{xcolor}
\usepackage{tikz}
\usetikzlibrary{arrows.meta,calc,angles,quotes}
\usepackage{subcaption}
\usepackage{hyperref}
\hypersetup{
  colorlinks=true,
  linkcolor=blue!42!black,
  citecolor=blue!42!black,
  urlcolor=blue!42!black,
  pdftitle={The longest-edge bisection algorithm may produce degenerating tetrahedra},
  pdfauthor={Sergey Korotov}
}

\newtheorem{theorem}{Theorem}[section]
\newtheorem{lemma}[theorem]{Lemma}
\theoremstyle{definition}

\DeclareMathOperator{\diam}{diam}
\DeclareMathOperator{\conv}{conv}
\newcommand{\R}{\mathbb{R}}

\definecolor{selectededge}{RGB}{170,45,35}
\definecolor{retainedface}{RGB}{45,95,155}
\definecolor{discarded}{RGB}{145,145,145}
\tikzset{
  tetedge/.style={line width=0.65pt,black},
  hiddenedge/.style={line width=0.55pt,dashed,gray!75},
  selected/.style={line width=1.45pt,selectededge},
  retainededge/.style={line width=0.95pt,retainedface!85!black},
  cutface/.style={fill=retainedface!18,draw=retainedface!75!black,line width=0.4pt,fill opacity=.55},
  vertex/.style={circle,fill=black,inner sep=1.05pt}
}

\title{The longest-edge bisection algorithm may produce degenerating tetrahedra}
\author{Sergey Korotov\\[-1pt]
\small Department of Business and Mathematics, M\"alardalen University, V\"aster\aa s, Sweden\\[-2pt]
\small \texttt{sergey.korotov@mdu.se, smkorotov@gmail.com}}
\date{August 24, 2026}

\begin{document}
\maketitle

\begin{abstract}
An explicit sequence of tetrahedra generated by the longest-edge bisection algorithm is shown to degenerate.  The example violates shape regularity and both the minimum- and maximum-angle conditions, demonstrating that arbitrary tie-breaking among longest edges does not guarantee nondegeneration.
\end{abstract}

\noindent\textbf{Keywords.} longest-edge bisection; tetrahedron; mesh degeneration; shape regularity; minimum-angle condition; maximum-angle condition.

\noindent\textbf{AMS classification.} 65M50, 65N50, 65N30.

\section{Introduction}

Bisection is a standard technique for constructing nested simplicial meshes.  The longest-edge rule is particularly natural because it is defined entirely by the geometry of the current simplex.  In two dimensions, convergence, lower-angle bounds, and the families of triangle shapes generated by repeated longest-edge bisection are well understood \cite{RosenbergStenger1975,Stynes1979,Stynes1980,Adler1983}.  Convergence questions have also been studied in arbitrary dimension, while marked-edge tetrahedral schemes can be designed to produce conforming locally refined meshes with only finitely many similarity classes \cite{Kearfott1978,ArnoldMukherjeePouly2000}.  Such marked constructions should be distinguished from the unrestricted geometric instruction considered here: at each step, bisect a currently longest edge and, when several edges are tied, allow any one of them to be selected.

The three-dimensional theory includes conforming face-to-face and generalized bisection procedures, regularity criteria for tetrahedral partitions, longest-edge $n$-section variants, and estimates for minimum and maximum angles \cite{KorotovKrizekKropac2008,HannukainenKorotovKrizek2010a,HannukainenKorotovKrizek2010b,HannukainenKorotovKrizek2014,KorotovPlazaSuarez2015,KorotovPlazaSuarez2016,KorotovPlazaSuarezMoreno2019,KorotovLundVatne2021}.  Other studies address special tetrahedral families, nondegeneracy, the number of generated similarity classes, and dynamical descriptions of refinement orbits \cite{PlazaPadronSuarezFalcon2004,PlazaPadronSuarez2005,PerdomoPlaza2014,SuarezTrujilloMoreno2021,PadronPlazaSuarez2023,TrujilloSuarezPadron2024,PadronTrujilloSuarez2025,MichaudKorotov2026a,MichaudKorotov2026b,MichaudKorotov2026c}.  These directions show that the behavior of tetrahedral refinement depends not only on the geometric selection rule, but also on the marking and tie-breaking conventions included in the algorithm.

The purpose of this note is to give a short exact counterexample for the unrestricted tetrahedral rule.  Its mechanism is a recurrent longest-edge tie.  The word ``may'' in the title is essential: the construction proves the existence of a degenerating admissible orbit, not degeneration for every prescribed tie-breaking convention.

\section{The two-step longest-edge bisection}

For $0<a\leq1$, let
\begin{equation}\label{eq:family}
E(a)=\conv\{A,B,C,D\}\subset\R^3,
\end{equation}
where
\begin{equation}\label{eq:vertices}
A=(0,0,a),\quad B=(0,0,0),\quad C=(\sqrt a,0,0),\quad
D=\left(\frac{\sqrt a}{2},\frac{\sqrt{7a}}2,0\right).
\end{equation}
The six squared edge lengths, in the order $(AB,AC,AD,BC,BD,CD)$, are
\begin{equation}\label{eq:Eedges}
\bigl(a^2,\ a+a^2,\ 2a+a^2,\ a,\ 2a,\ 2a\bigr).
\end{equation}
Thus $AD$ is the unique longest edge.  Let
\[
M=\frac{A+D}{2}
\]
and retain, after bisecting $AD$, the child
\[
O(a)=\conv\{M,B,D,C\}.
\]
Its squared edge lengths, in the order $(MB,MD,MC,BD,BC,DC)$, are
\begin{equation}\label{eq:Oedges}
\left(\frac a2+\frac{a^2}{4},\ \frac a2+\frac{a^2}{4},\
 a+\frac{a^2}{4},\ 2a,\ a,\ 2a\right).
\end{equation}
Since $a+a^2/4<2a$, the two edges $BD$ and $DC$ are exactly tied for longest.

Choose $BD$, put $N=(B+D)/2$, and retain
\[
F(a)=\conv\{M,N,B,C\}.
\]

\begin{lemma}[Exact recurrence]\label{lem:recurrence}
After the relabeling
\[
A'=M,\qquad B'=N,\qquad C'=B,\qquad D'=C,
\]
the tetrahedron $F(a)$ is congruent to $E(a/2)$.  Consequently,
\begin{equation}\label{eq:recurrence}
\boxed{
E(a)\xrightarrow[\mathrm{retain}\ MBDC]{\mathrm{bisect}\ AD}
O(a)\xrightarrow[\mathrm{retain}\ MNBC]{\mathrm{bisect}\ BD}
E(a/2).}
\end{equation}
\end{lemma}

\begin{proof}
The squared edge lengths of $F(a)$ after the indicated relabeling are
\[
\left(\frac{a^2}{4},\ \frac a2+\frac{a^2}{4},\ a+\frac{a^2}{4},\
\frac a2,\ a,\ a\right).
\]
With $a'=a/2$, this is precisely
\[
\bigl((a')^2,\ a'+(a')^2,\ 2a'+(a')^2,\ a',\ 2a',\ 2a'\bigr),
\]
which is the edge sextuple \eqref{eq:Eedges} for $E(a')$.
\end{proof}

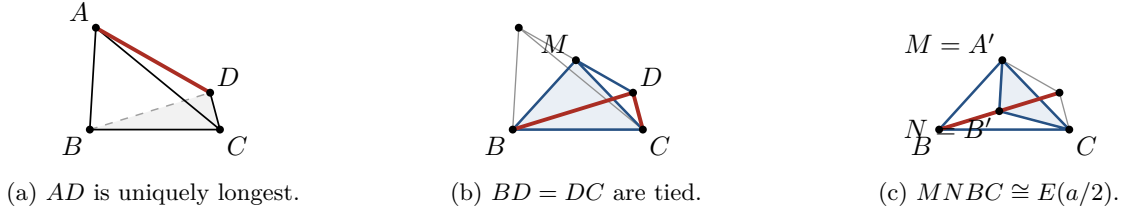
\begin{figure}[t]
\centering
\begin{subfigure}[t]{0.31\textwidth}
\centering
\begin{tikzpicture}[x={(2.15cm,0cm)},y={(0.69cm,0.46cm)},z={(0.10cm,1.68cm)},scale=.80]
  \coordinate (A) at (0,0,1);
  \coordinate (B) at (0,0,0);
  \coordinate (C) at (1,0,0);
  \coordinate (D) at (.5,1.322876,0);
  \fill[gray!10] (B)--(C)--(D)--cycle;
  \draw[hiddenedge] (B)--(D);
  \draw[tetedge] (B)--(C)--(D) (A)--(B) (A)--(C);
  \draw[selected] (A)--(D);
  \foreach \P in {A,B,C,D}{\node[vertex] at (\P) {};}
  \node[above left=-1pt] at (A) {$A$};
  \node[below left=-1pt] at (B) {$B$};
  \node[below right=-1pt] at (C) {$C$};
  \node[above right=-1pt] at (D) {$D$};
\end{tikzpicture}
\caption{$AD$ is uniquely longest.}
\end{subfigure}\hfill
\begin{subfigure}[t]{0.31\textwidth}
\centering
\begin{tikzpicture}[x={(2.15cm,0cm)},y={(0.69cm,0.46cm)},z={(0.10cm,1.68cm)},scale=.80]
  \coordinate (A) at (0,0,1);
  \coordinate (B) at (0,0,0);
  \coordinate (C) at (1,0,0);
  \coordinate (D) at (.5,1.322876,0);
  \coordinate (M) at (.25,.661438,.5);
  \fill[cutface] (M)--(B)--(C)--cycle;
  \draw[discarded,line width=.5pt] (A)--(B)--(C)--(A) (A)--(M);
  \draw[hiddenedge] (B)--(D);
  \draw[retainededge] (M)--(B) (M)--(C) (M)--(D) (B)--(C);
  \draw[selected] (B)--(D)--(C);
  \foreach \P in {A,B,C,D,M}{\node[vertex] at (\P) {};}
  \node[above left=-1pt] at (M) {$M$};
  \node[below left=-1pt] at (B) {$B$};
  \node[below right=-1pt] at (C) {$C$};
  \node[above right=-1pt] at (D) {$D$};
\end{tikzpicture}
\caption{$BD=DC$ are tied.}
\end{subfigure}\hfill
\begin{subfigure}[t]{0.31\textwidth}
\centering
\begin{tikzpicture}[x={(2.15cm,0cm)},y={(0.69cm,0.46cm)},z={(0.10cm,1.68cm)},scale=.80]
  \coordinate (B) at (0,0,0);
  \coordinate (C) at (1,0,0);
  \coordinate (D) at (.5,1.322876,0);
  \coordinate (M) at (.25,.661438,.5);
  \coordinate (N) at (.25,.661438,0);
  \fill[cutface] (M)--(N)--(C)--cycle;
  \draw[discarded,line width=.5pt] (M)--(D)--(C) (N)--(D);
  \draw[hiddenedge] (B)--(D);
  \draw[retainededge] (M)--(B) (M)--(C) (M)--(N) (N)--(B) (N)--(C) (B)--(C);
  \draw[selected] (B)--(N)--(D);
  \foreach \P in {B,C,D,M,N}{\node[vertex] at (\P) {};}
  \node[below left=-1pt] at (B) {$B$};
  \node[below right=-1pt] at (C) {$C$};
  \node[above left=-1pt] at (M) {$M=A'$};
  \node[below left=-1pt] at (N) {$N=B'$};
\end{tikzpicture}
\caption{$MNBC\cong E(a/2)$.}
\end{subfigure}
\caption{The two steps of the bisection.  Red edges are selected or tied longest edges; the retained child is emphasized in blue.}
\label{fig:orbit}
\end{figure}

Starting with $E(1)$ and iterating Lemma~\ref{lem:recurrence} produces
\begin{equation}\label{eq:chain}
E(1)\longrightarrow O(1)\longrightarrow E(1/2)\longrightarrow O(1/2)
\longrightarrow E(1/4)\longrightarrow\cdots.
\end{equation}
Every arrow is a valid longest-edge bisection.  Only the arrows leaving $O(a)$ require a choice from a tie.

\section{Violation of regularity and angle conditions}

For a tetrahedron $T$, write $h_T=\diam T$.  Shape regularity requires a constant $c>0$ such that $|T|\geq c h_T^3$ throughout the family; equivalent regularity criteria are discussed in \cite{BrandtsKorotovKrizek2008}.  The tetrahedral minimum-angle condition requires all face and interior dihedral angles to be bounded below by a positive constant, while the maximum-angle condition requires them to be bounded above by a constant smaller than $\pi$ \cite{Krizek1992,KorotovVatne2020,KorotovKrizek2024}.

Let $a_k=2^{-k}$, $E_k=E(a_k)$, and $O_k=O(a_k)$.  The base triangle $BCD$ has area $\sqrt7a/4$, and the height from $A$ to the base plane is $a$.  Hence
\begin{equation}\label{eq:volume}
|E(a)|=\frac{\sqrt7}{12}a^2,
\qquad
h_{E(a)}^2=|AD|^2=a(2+a).
\end{equation}

\begin{theorem}[Three simultaneous failures]\label{thm:failure}
The admissible sequence \eqref{eq:chain} violates shape regularity, the tetrahedral minimum-angle condition, and the tetrahedral maximum-angle condition.  More precisely,
\begin{align}
\frac{|E_k|}{h_{E_k}^3}
&=\frac{\sqrt7}{12}\frac{\sqrt{a_k}}{(2+a_k)^{3/2}}
\longrightarrow0, \label{eq:regfail}\\
\delta(a_k)
&=\arctan\sqrt{\frac{8a_k}{7}}
\longrightarrow0, \label{eq:minfail}\\
\cos\Theta(a_k)
&=\frac{2a_k-7}{\sqrt{(4a_k+7)(8a_k+7)}}
\longrightarrow-1, \label{eq:maxfail}
\end{align}
where $\delta(a)$ is the interior dihedral angle of $E(a)$ at $CD$, and $\Theta(a)$ is the interior dihedral angle of $O(a)$ at $MC$.  Thus $\Theta(a_k)\to\pi$.
\end{theorem}

\begin{proof}
Formula \eqref{eq:regfail} follows directly from \eqref{eq:volume}; its right-hand side is asymptotic to
\[
\frac{\sqrt7}{12\,2^{3/2}}\,2^{-k/2}.
\]
Therefore no uniform positive lower bound for $|T|/h_T^3$ is possible.

To obtain \eqref{eq:minfail}, intersect $E(a)$ with a plane perpendicular to $CD$.  Since
\[
d(B,CD)=\frac{2|BCD|}{|CD|}=\sqrt{\frac{7a}{8}}
\]
and $A$ lies vertically a distance $a$ above $B$, the relevant cross-sectional angle is
\[
\delta(a)=\arctan\frac{a}{\sqrt{7a/8}}
=\arctan\sqrt{\frac{8a}{7}}.
\]
Hence an interior dihedral angle tends to zero.

For the angle at $MC$ in $O(a)$, let $u=C-M$ and project $B-M$ and $D-M$ onto the plane perpendicular to $u$.  Calling the projected vectors $p_B$ and $p_D$, direct calculation gives
\[
|p_B|^2=\frac{a(4a+7)}{4(a+4)},\qquad
|p_D|^2=\frac{a(8a+7)}{4(a+4)},\qquad
p_B\!\cdot p_D=\frac{a(2a-7)}{4(a+4)}.
\]
Their angle is the interior dihedral angle $\Theta(a)$, which proves \eqref{eq:maxfail}.  As $a\to0$, its cosine tends to $-1$, and therefore $\Theta(a)\to\pi$.
\end{proof}

In fact,
\begin{equation}\label{eq:rates}
\delta(a)\sim\sqrt{\frac87}\sqrt a,
\qquad
\pi-\Theta(a)\sim\frac4{\sqrt7}\sqrt a.
\end{equation}
Thus the normalized-volume ratio and both angular gaps decay at the same geometric order $2^{-k/2}$.

\section{Tie-breaking and conforming partitions}

The exact equality $BD=DC$ in every $O(a)$ is the decisive feature.  Consequently, the example establishes
\[
\boxed{\text{The rule ``choose any longest edge'' does not guarantee regularity in three dimensions.}}
\]
It does not assert that every fixed tie-breaking convention follows this orbit.  Nevertheless, the bad branch is selected by a simple deterministic rule: among tied longest edges, choose the one whose opposite edge is longer.  In $O(a)$, the edge opposite $BD$ is $MC$, while the edge opposite $DC$ is $MB$, and
\[
|MC|^2=a+\frac{a^2}{4}>\frac a2+\frac{a^2}{4}=|MB|^2.
\]

The local chain can also be scheduled inside the global face-to-face longest-edge algorithm.  Before selecting the next prescribed edge of length $L$, repeatedly bisect globally longest edges longer than $L$.  This waiting process is finite: if a globally longest edge has length $c$, every new midpoint-to-opposite-vertex edge has length at most $(\sqrt3/2)c$.  Once the global maximum equals $L$, the prescribed edge may be chosen from the tie.  Repetition yields conforming partitions containing all tetrahedra $E_k$ and $O_k$, so the same three failures occur at the partition level.

\section{Conclusion}

The family \eqref{eq:family} supplies a closed-form two-bisection recurrence that halves a shape parameter while retaining admissibility under the longest-edge rule.  Along the resulting infinite sequence, the normalized volume tends to zero, one dihedral angle tends to zero, and another tends to $\pi$.  Therefore tetrahedral longest-edge bisection may produce degenerating elements unless tie resolution is included explicitly in the algorithm and analyzed as part of the regularity theorem.

\end{document}